\documentclass[reqno,10pt]{amsart}
\usepackage{amssymb,amsmath,amsthm,enumerate,amscd,graphicx,color}
\usepackage[usenames,dvipsnames,svgnames,table]{xcolor}
\usepackage{enumitem}
\usepackage{pdfsync}
\usepackage{float}
 \usepackage[colorlinks=true, pdfstartview=FitV, linkcolor=blue, citecolor=blue, urlcolor=blue]{hyperref}
\newcommand{\bC}{{\mathbb C}}
\newcommand{\bZ}{{\mathbb Z}}

\newcommand{\End}{\operatorname{End}}

\newcommand{\fH}{\mathcal H}

\newcommand{\fA}{\mathfrak A}

\newcommand{\fgl}{\mathfrak {gl} }

\begin{document}

\theoremstyle{definition}
\theoremstyle{remark}
\newtheorem{definition}{Definition}[section]
\newtheorem{remark}{Remark}[section]
\newtheorem{example}{Example}[section]
\newtheorem*{notation}{Notation}

\theoremstyle{rem}
\numberwithin{equation}{section}

\newtheorem{theorem}{Theorem}[section]
\newtheorem{corollary}[theorem]{Corollary}
\newtheorem{lemma}[theorem]{Lemma}
\newtheorem{proposition}[theorem]{Proposition}
\newtheorem*{theorem*}{Theorem}

\author{Alexei Borodin}
\address{Department of Mathematics, Massachusetts Institute of Technology, 77 Massachusetts ave. Cambridge, MA 02139, USA}
\email{borodin@math.mit.edu }

\author{Natasha Rozhkovskaya}
\address{Department of Mathematics, Kansas State University, Manhattan, KS 66502, USA}
\email{rozhkovs@math.ksu.edu}

\thanks{}

\begin{abstract}
Two super-analogs of the Schur\,-\,Weyl duality are considered: the duality of actions in $(\bC^{m|n})^{\otimes N}$ of the Lie superalgebra $\fgl(m,n)$ and the symmetric group $S_N$, and the duality of actions  of the Lie superalgebra $Q(n)$ and  a certain  finite group $Se(N)$ 
in $(\bC^{n|n})^{\otimes N}$. We construct an isomorphism  of symmetric and universal enveloping algebras of these Lie superalgebras called special symmetrization. Using this isomorphism of vector spaces we describe explicitly  the duality between the centers of the corresponding  universal enveloping algebras and the group algebras. 
\end{abstract}

\title {On a super\,-\,analog of the Schur\,-\,Weyl Duality}

\maketitle

\setcounter{section}{-1}

\section{Introduction}\label{Sec-0}

The results of the present note were obtained in 1995 and written up as a preprint of the 
Erwin Sch\"{o}dinger Institute in Vienna. They were inspired by the work of G.I.\,Olshanski \cite{Olsh-1}
on the so-called \emph{special symmetrization}. Further results on special symmetrization
were obtained in \cite{OkOl-1} and \cite{Olsh-2}; however, the case of Lie superalgebras that we dealt with seems to have so far 
remained uncovered by the existing literature. For that reason, and also because of the continuing interest of researchers to the structure 
 of centers of  universal enveloping algebras of Lie superalgebras (see, e.g., 
\cite{AlSaSa, Cheng-Wang, Naz, SaSa, SaSaSe}), we have decided to make our work publicly available. 
The text below is a lightly edited version of the 1995 preprint.

Let $U(\fgl(m,n))$ be the universal enveloping algebra of the Lie superalgebra $\fgl(m,n)$, and let $\bC[S_N]$ be the group algebra of the symmetric group $S_N$. These two associative superalgebras  (all the elements of the second one are even) act on the tensor space
 $(\bC^{m|n})^{\otimes N}$. For  $U(\fgl(m,n))$  it is the $N$-th tensor power of the vector representation, and  $\bC[S_N]$   permutes the components $(\bC^{m|n})$ of the tensor product. It was proved in \cite{Serg-1},
see also \cite{Ber-Reg} in the case of $\mathfrak{gl}(m,n)$,
that these two actions are dual in the following sense: the commutant of the image of $U(\fgl(m,n))$ in $\End (\bC^{m|n})^{\otimes N} $ coincides with the image of $\bC[S_N]$,  and {\it vice  versa}. This duality  is a super-analog of the well-known  Schur\,-\,Weyl duality  for  $\fgl(n)$ and $S_N$ actions on  $(\bC^n)^{\otimes N}$.
 
 The matrix Lie superalgebra 
 \begin{align}
 	Q(n)=
	\left\{
	 	\begin{pmatrix}
	 	A& B\\
	 	B&A
	 	\end{pmatrix} \vert\, A,B \in \text{Mat}\, (n,\bC)
	\right\}\subset \fgl(n,n)
 \end{align}
  has an irreducible vector representation. A dual group  for  the  action of  this Lie superalgebra  in  $(\bC^{n|n})^{\otimes N}$
   is a double covering of the wreath product of  $\bZ_2^N$ and $S_N$ with the natural action of $S_N$ on  $\bZ_2^N$
  \cite{Serg-1}. We refer to this group as the {\it Sergeev group} and denote by $Se(N)$ (see a  precise definition  in  Section \ref {Sec-2}). 
We refer to \cite{Cheng-Wang} for a  detailed exposition on the  Schur-Sergeev duality.
  
  The above dualities  imply  that the centers of $U(\fgl(m,n))$ and $\bC[S_N]$, denoted  by $Z(\fgl(m,n))$ and $Z(S_N)$ respectively,  as well as the centers of $U(Q(n))$ and $\bC[Se(N)]$, denoted  by $Z(Q(n))$ and $Z(Se(N))$ respectively, have the same images in the automorphisms of the corresponding tensor spaces.    The problem is to find explicitly an element of $Z(\fgl(m,n))$ (or  $Z(Q(n))$)  that acts on   a tensor space identically with a given central element from the group algebra  of the dual group. This problem was solved in \cite{Ker-Olsh} for the classical Schur\,-\,Weyl duality. The result has important applications  to the theory of symmetric functions and to the classification of  characters  of the infinite symmetric  group. The key to solving the  ``classical'' problem   is in introducing a suitable map called {\it special symmetrization}  that establishes an isomorphism  of $\fgl(n)$-modules $S(\fgl(n))$  (symmetric algebra of the vector space $\fgl(n)$) and 
    $U(\fgl(n))$. The special symmetrization  differs from the usual symmetrization,  and  for the first time it appeared in   \cite{Olsh-1}
    in the study of   classical infinite-dimensional  Lie groups  and  corresponding  analogues of universal  enveloping algebras. Generalizations of  special symmetrization and its versions for classical Lie algebras  were constructed in \cite{OkOl-1, Olsh-2}.
    
    In this note we  define the special symmetrization in the super-case  for $\fgl(m,n)$ and $Q(n)$. Moreover, using this map we define elements of $Z(\fgl(m,n))$ and $Z(Q(n))$ that coincide in $\End (\bC^{m|n})^{\otimes N} $  with the image of natural basis elements 
  of $Z(S_N)$ and $Z(Se(N))$. Our main results are Theorems \ref{thm_2.4} and \ref{thm_2.5} below.

 \subsection*{Acknowledgments}
 The authors are grateful to the Erwin Shr\"odinger Institute in Vienna for the perfect conditions for finishing this work. The work was supported in part by the Soros International Science Foundation. The authors  are very grateful to G.\,I.\,Olshanski for the setting of the problem and for constant attention to the work. 
 
   \section{Special Symmetrization} \label{Sec-1}
   Let $S(\fgl(m,n))$  denote the symmetric algebra, let  $T(\fgl(m,n))$ be the tensor algebra, and  let  $U(\fgl(m,n))$ be the universal enveloping algebra of the Lie superalgebra $\fgl(m,n)$. Define a linear map  $S(\fgl(m,n))\otimes \fgl(m,n)\, \to \, S(\fgl(m,n))$, sending $A\otimes X$ to an element $A*X\in S(\fgl(m,n)) $ as follows. For homogeneous elements  $Y_1, \dots, Y_k$ and $X$  we set
   \begin{align}\label{eq_1.1}
   	(Y_1\dots Y_k) * X= 	Y_1\dots Y_k X +\sum_{r=1}^{k} (-1)^{P(r)} Y_1\dots Y_{r-1}\langle XY_r\rangle Y_{r+1}\dots Y_k.
   \end{align}
 Here $\langle XY\rangle $ denotes the product in the associative matrix algebra and 
 \[
 P(r)=\sum_{s=r+1}^{k} p(Y_s)\, p(X),
 \]
 where $p$ is the parity  function on the super space $\fgl(m,n)$.
 
 Consider the linear map
 \begin{align}\label{eq_1.2}
 \begin{split}
 	 T(\fgl(m,n))\, &\to \, S(\fgl(m,n)),\\
	 X_1\otimes \dots \otimes X_k \, & \mapsto\,  ((\dots (X_1* X_2)* \dots ) *X_{k-1})*X_k.	 
\end{split}	 
 \end{align}
 \begin{lemma}\label{lem_1.1}
 The map (\ref{eq_1.2}) factors through to  a  linear map $\tilde \sigma: U(\fgl(m,n))\to S(\fgl(m,n))$.
 \end{lemma}
 \begin{proof}
 We have to check  that for any monomials $A,B$ and  homogeneous  elements $X,Y\in \fgl(m,n)$ the image of $A\otimes (X\otimes Y-
 (-1)^{p(X)p(Y)}Y\otimes X)\otimes B$ is equal to the image of $A\otimes [X,Y]\otimes B$. This is verified by a direct calculation using the identity $\langle X Y\rangle - (-1)^{p(X)p(Y)}\langle YX\rangle =[X,Y]$.
  \end{proof}
 \begin{proposition}\label{prop_1.2}
 The  map $\tilde \sigma$ from Lemma \ref{lem_1.1} is invertible. 
 \end{proposition}
 \begin{proof}
 Let $\{E_{ij}\}$ be the standard matrix units, and   let $U_k$ (or $V_k$) be a linear span of monomials $\{ E_{i_1j_1}\dots E_{i_mj_m}, m\le k\}$ in $U(\fgl (m, n))$ (or  in $S(\fgl(m,n))$). From (\ref{eq_1.1}) and (\ref{eq_1.2})  we obtain 
 \[
 \tilde \sigma (E_{i_1,j_1}\dots E_{i_m j_m}) = E_{i_1j_1}\dots E_{i_m j_m} +A,\quad A\in V_{m-1}.
 \]
 Hence $\tilde \sigma(U_k)\subseteq V_k$, and in the {basis} $\{E_{i_1j_1}\dots E_{i_mj_m}, m\le k\}$ the map $\tilde \sigma\vert _{U_k}$ is represented by a  triangular matrix with ones on the main diagonal. This means that  for every $k=0,1,\dots, $ the map $\tilde \sigma\vert _{U_k}$ is invertible, and therefore $\tilde \sigma$ is invertible on the whole space $U(\fgl(m,n))$.
 \end{proof}
 
 Thus, $\tilde \sigma$ defines a linear isomorphism $\sigma =\tilde \sigma ^{-1}: S(\fgl(m,n))\to U(\fgl(m,n))$. This  isomorphism  can be regarded as an analog of the special symmetrization  introduced in \cite{Ker-Olsh}, see also \cite{Olsh-1}.
  
  We shall obtain an explicit formula for the map $\tilde \sigma$. 
  The image of a monomial $E_{i_1 j_1}\dots E_{i_Mj_M}\in U(\fgl(m,n))$ under $\tilde \sigma$  is equal to the sum of all monomials in $S(\fgl(m,n))$ that can be obtained as follows. Fix a partition of the set $\{1,\dots, M\}$ into disjoint subsets. Let us group matrix units corresponding to each subset together  and multiply  matrices  in each subset. The result of application of $\tilde \sigma$ is a sum  of  basis  monomials in $S(\fgl(m,n))$, each  supplied with a coefficient $\pm 1$.
  
  To formalize the above description we need some notation. 
  Let $\alpha$ be a partition  of the set $\{1,\dots, M\}$ into  subsets (blocks):
  \[
  	\alpha= 
	 \{ m_1\dots  l_1\} \dots
	 \{ m_r\dots l_r\}.
  \]
 A partition $\alpha $ is called {\it regular}  if elements in every block   are ordered as 
 $m_q<\dots<l_q,\, q=1,\dots, r$, and if  the blocks are listed in the ascending order of their first elements:
 $m_1<\dots <m_s<\dots <m_r,\, 1<\dots <s<\dots <r$.
 
 Consider a regular partition $\alpha= \{ m_1\dots l_1\} \dots \{ m_r\dots l_r\}$. For $t=1,\dots, r$, set 
 \begin{align*}
 	\delta_t(\alpha)=
		\begin{cases}
		\prod_{k=1}^{d_t-1}\delta_{j_{s_k}}  ^{i_{s_{k+1}}}, & \text{if}\quad d_t>1,\\
		1, & \text{if}\quad d_t=1,
		\end{cases}
 \end{align*}
 where $\delta^i_j$ is the Kronecker delta,  $\{s_1=m_t \dots  s_{d_t}=l_t\}$ is the block number $t$ in the partition $\alpha$,  and $d_t$ is the length of this block. Set
  \begin{align*}
 	\delta(\alpha)=\prod_{t=1}^{r}\delta _t(\alpha).
  \end{align*}
   Introduce  functions  $\theta_k(s), \eta_k(s)$, $k=1,\dots, r$, such that  for $s\in\{1,\dots, M\}$ we have
 \begin{align*}
 	\theta_k(s)&=
		\begin{cases} \max\{u<s,\, u\in \{m_k,\dots, l_k\}\},
		 & \text{if the  $k$-th block contains a number less than $s$};\\
		0, & \text{otherwise;}
		\end{cases}
\\
 	\eta_k(s) &=
		\begin{cases}
		m_k, & \text{if}\quad m_k<s;\\
		0, & \text{otherwise.}
		\end{cases}
 \end{align*}
 With $p$ being the parity function on $\fgl(m,n)$, we now set $p_{ij}=p(E_{ij})$ and  assume that $p_{i_mj_k}=0$ for $k=0$ or $m=0$. Define
  \begin{align*}
  	q_s(\alpha)&= p_{i_sj_s}\sum_{k>t} p_{i_{\eta_k(s)} j_{\theta_k(s)}},\quad \text{for}\quad s\in\{m_t,\dots, l_t\},\\
  	q(\alpha)&=\sum_{s=1}^{M}q_s(\alpha).
  \end{align*}
  \begin{proposition}\label{prop_1.3}
   Let $E_{i_1, j_1}\dots E_{i_Mj_M}$ be a basis element of $U(\fgl(m,n))$. Then 
  \begin{align}\label{eq_1.3}
  	\tilde \sigma(E_{i_1, j_1}\dots E_{i_Mj_M})=\sum_{\alpha}(-1)^{q(\alpha)}\delta(\alpha) E_{i_{m_1}j_{l_1}}\dots E_{i_{m_r}j_{l_r}},
  \end{align}
  where the summation is taken over the set of all regular partitions  $\alpha= \{ m_1\dots l_1\} \dots \{ m_r\dots l_r\}$ of
   $\{1,\dots, M\}$.
  \end{proposition}
  \begin{proof}
  Induction on $M$.
  \begin{enumerate}
  	\item  For $M=2$
 \begin{align*}
 	 \tilde \sigma(E_{i_1 j_1}E_{i_2j_2})=E_{i_1, j_1}*E_{i_2j_2} = E_{i_1, j_1}E_{i_2j_2} +\delta _{j_1}^{i_2}E_{i_1 j_2}.
  \end{align*}
  	\item By the  inductive hypothesis we have 
\begin{align*}
  	\tilde \sigma(E_{i_1 j_1}\dots E_{i_{M}j_{M} }) =\sum_\beta (-1)^{q(\beta)} \delta(\beta)E_{i_{m_1} j_{l_1} }\dots E_{i_{m_r}j_{l_r}},
\end{align*}	
where $\beta=\{m_1\dots l_1\}\dots \{ m_r\dots l_r\}$ is a regular partition  of $\{1,\dots, M\}$. Then
\begin{align}\label{eq_1.4}
\begin{split}
  	\tilde \sigma &(E_{i_1 j_1}\dots E_{i_{M+1}j_{M+1}})
      				=\tilde \sigma(E_{i_1 j_1}\dots E_{i_{M}j_{M}}) * E_{i_{M+1}j_{M+1}}
	\\
	&=\sum_{\beta}(-1)^{q(\beta)}\delta(\beta) E_{i_{m_1}j_{l_1}}\dots E_{i_{m_r}j_{l_r}} E_{i_{M+1}j_{M+1}} 
	\\
	&\quad +\sum_{\beta}\sum_{t=1}^r (-1)^{q(\beta)+\tilde p (t)}\delta(\beta)\,\delta_{j_{l_t}}^{i_{M+1}}\, 
	E_{i_{m_1}j_{l_1}}\dots E_{i_{m_t}j_{M+1}} E_{i_{m_r}j_{l_r}},
\end{split}	
 \end{align}
with    $\tilde p(t)=\sum_{k=t+1}^{r} p_{i_{m_k}j_{l_k}}p_{i_{M+1}j_{M+1}}$.

     Every regular partition of the set $\{1,\dots, M+1\}$  can be obtained from a regular partition of the set $\{1,\dots, M\}$ 
     by inserting $M+1$ at the end of one of the blocks.  Thus, (\ref{eq_1.4})  is a sum of monomials $E_{i_{m_1^\prime} \,j_{l_1^\prime}}\dots E_{i_{m^\prime_{r^\prime}}\,j_{_{l^\prime_{r^\prime}}}}$  taken with some coefficients,  where the summation is over all regular  partitions 
     $\{m^\prime_1 \dots l^\prime_1\}\dots \{m^\prime_{r^\prime}\dots l^\prime_{r^\prime}\}$ of the set  $\{1,\dots, M+1\}$.
     
     Let us compute the coefficients of the monomials. For 
 \[
     	\alpha=\{m_1\dots l_1\}\dots \{m_r\dots l_r\}\{M+1\}
	 \]
     and 
 \[
     	\beta=\{m_1\dots l_1\}\dots \{m_r\dots l_r\}
	\]
     we have 
 \begin{align}\label{eq_1.5}
 \begin{split}
     	q(\alpha)&=q_{M+1}(\alpha)+\sum_{k=1}^{M}q_k(\alpha)=q(\beta),\\
     	\delta(\alpha)&=\delta_{r+1}(\alpha) \prod_{k=1}^r\delta_k(\alpha)=\delta_{r+1}(\alpha)\delta(\beta)=\delta(\beta).
\end{split}	
\end{align}
     For 
 \[
     	\alpha=\{m_1\dots l_1\}\dots \{m_t\dots l_t, M+1\} \dots \{m_r\dots l_r\}\
	 \]
      and 
 \[
     	\beta=\{m_1\dots l_1\}\dots \{m_r\dots l_r\}
	\]
     we have 
 \begin{align*}
     	q_{M+1}(\alpha)&=\sum_{k=t+1}^r p_{i_{M+1}j_{M+1}}p_{i_{m_k}j_{l_k}},\\
	q_s(\alpha)&=q_{s}(\beta) \quad \text{  for every $s\in \{1,\dots, M\}$.}
 \end{align*}
     Therefore, 
 \begin{align}\label{eq_1.6}
 \begin{split}
 	q(\alpha)&= q(\beta) + q_{M+1}(\alpha) = q(\beta)+ \tilde p(t),\\
	  \delta(\alpha)&=\delta^{i_{M+1}}_{j_{l_t}}\delta(\beta).
\end{split}
 \end{align}
 The claim of the  proposition follows from (\ref{eq_1.4}), (\ref{eq_1.5}), and (\ref{eq_1.6}) immediately. 
  \end{enumerate}
  \end{proof}
  
From  Proposition \ref{prop_1.3} it is easy to deduce that $\tilde \sigma $ is an isomorphism not only of vector spaces, but  of $\fgl(m,n)$-modules $U(\fgl(m,n))$ and $S(\fgl(m,n))$.
  
  Let us consider the space of polynomials in $(m+n)^2$ variables $\{x_{ij}\}$. The parity of $x_{ij}$  is set to be equal to  the parity of $E_{ij}$. We denote by $\mathcal D$ the algebra of polynomial differential operators acting  on this space of polynomials. Define a map $\partial$ from 
  $\fgl(m,n)$ to $\mathcal D$  by   $\partial: E_{ij}\mapsto \sum_{k=1}^{m+n} x_{ki}\partial_{kj}$. One can easily check  that the map $\partial$ is a homomorphism of  Lie superalgebras. Therefore, it can be extended to a homomorphism $\partial: U(\fgl(m,n)) \to \mathcal D$. It can be proved that $Ker \,\partial =\{0\}$. The construction is similar to the realization of a universal enveloping algebra by left-invariant differential operators on a Lie group. Thus, 
\begin{align}\label{eq_1.7}
	\partial (E_{i_1j_1}\dots E_{i_lj_l})= \left(  \sum_{k_1}x_{k_1i_1}\partial _{k_1 j_1}\right)\dots  \left( \sum_{k_l}x_{k_li_l}\partial _{k_l j_l}\right).
\end{align}
 Now we shall give a formula for the special symmetrization in terms of differential operators.
  \begin{proposition}\label{prop_1.4}
Let  $\tilde \partial =\partial \circ \sigma$. Then for $E_{i_1j_1}\dots E_{i_Mj_M}\in S(\fgl(m,n))$ we have 
  \begin{align}\label{eq_1.8}
  \tilde \partial (E_{i_1j_1}\dots E_{i_Mj_M})= 
 \sum_{\overline k=(k_1\dots k_{M})}(-1)^{p(\overline k)}x_{k_1 i_1}\dots x_{k_Mi_M}\partial _{k_1j_1}\dots \partial _{k_Mj_M},
  \end{align}
  where $p(\overline k)=\sum_{1\le q< t\le M}p_{k_qj_q}p_{k_ti_t}$.
  \end{proposition}
  \begin{proof}
  It suffices to verify the equality $\tilde \partial \circ \tilde \sigma=\partial$ for the map $\tilde \partial$ defined by (\ref{eq_1.8}). We now prove that 
\begin{align}\label{eq_1.9}
	\sum_{\overline k} x_{k_1i_1}\partial_{k_1j_1} \dots x_{k_Mi_M}\partial_{k_Mj_M} = \sum_{\alpha}\sum_{\overline k}
	(-1)^{p_{\overline k}(\alpha)+q(\alpha)}\delta (\alpha) D_{\overline k}(\alpha),
\end{align}
where 
\[
	D_{\overline k}(\alpha) = D_{\overline  k; i_{m_1}\dots i_{m_r} j_{l_1}\dots j_{l_r}} =    x_{k_1i_{m_1}}\dots x_{k_ri_{m_r}}\partial_{k_1j_{l_1}} \dots \partial_{k_rj_{l_r}},
	\]
 $\alpha=\{m_1\dots l_1\}\dots \{m_r\dots l_r\}$ is a regular  partition of the set $\{1,\dots, M\}$,  $ \overline  k=(k_1,\dots, k_r)$, and
\[
	p_{\overline k}(\alpha)=\sum_{1\le q<t\le r} p_{k_{q}j_{l_q}} p_{k_ti_{m_t}}.
	\]
The conclusion of Proposition $\ref{prop_1.4}$   follows from (\ref{eq_1.9}) and Proposition \ref{prop_1.3}. It is easy to check  (\ref{eq_1.9}) for $M=2$, hence
\[
	\partial(E_{i_1j_1}E_{i_2j_2})
	 = \tilde \partial \left(  \delta_{j_1}^{i_2}E_{i_1j_2}+ E_{i_1j_1} E_{i_2j_2}\right) =\tilde \partial \circ \tilde \sigma (E_{i_1j_1}E_{i_2j_2}).
	\]
Using induction on $M$ again and applying $\sum_{\overline k} D_{\overline k}(\alpha)$ to $\sum_{k_{M+1}}x_{k_{M+1}i_{M+1}}\partial_{k_{M+1}j_{M+1}}$ we obtain

\begin{align}
\label{eq_1.10}
\begin{split}
	 \sum_{\overline k} D_{\overline k}(\alpha)\,\sum_{k_{r+1}}x_{k_{r+1}i_{M+1}}\partial_{k_{r+1}j_{M+1}}
	 &=\sum_{\overline k } \sum_{t=1}^r\delta_{j_{l_t}}^{i_{M+1}} (-1)^{p_{\overline k}^\prime} D_{\overline k; i_{m_1}\dots i_{m_r}j_{l_1}\dots j_{l_{t-1}}j_{M+1}j_{l_{t+1}}\dots j_{l_r} }\\
	 &+\sum_{(\overline k, k_{r+1})} (-1)^{p^{\prime\prime}_ {(\overline k, k_{r+1}) }} D_{(\overline k, k_{r+1}); i_{m_1}\dots i_{m_r}i_{M+1}j_{l_1}\dots j_{l_r} j_{M+1}} ,
\end{split}	 
\end{align}
where
\begin{align*}
	p^\prime_{\overline k} =\sum_{s=t+1}^r p_{k_s j_{l_s}} p_{i_{M+1}j_{M+1}},\quad 
	p^{\prime\prime}_{(\overline k, k_{r+1})} =\sum_{t=1}^r p_{k_t j_{l_t}} p_{k_{r+1}i_{M+1}}.
\end{align*}
For the same reason as in the proof of Proposition \ref{prop_1.3} the sum 
\begin{align*}
	\sum_\alpha \sum_{\overline k}(-1)^{p_{\overline k}(\alpha) +q(\alpha)}\delta(\alpha) D_{\overline k}(\alpha) \sum_{k_{M+1}}x_{k_{M+1}i_{M+1}}\partial _{k_{M+1}j_{M+1}}
\end{align*}
is the sum of monomials  $D_{\overline k}(\beta)$ with some  coefficients  over all regular partitions $\beta$ of  $\{1,\dots, M+1\}$.
Note that for 
\begin{align*}
	\beta&=\{m_1\dots l_1\}\dots \{m_r\dots l_r\}\{M+1\},
\\
	\alpha&=\{m_1\dots l_1\}\dots \{m_r\dots l_r\}
\end{align*}
we have 
\begin{align}\label{eq_1.11}
\begin{split}
	\delta(\beta)&=\delta(\alpha), \quad q(\beta)= q(\alpha),\\
	p_{\overline k}(\beta)&=p_{\overline k}(\alpha) +\sum_{s=1}^{r} p_{k_sj_{l_s}}p_{k_{r+1} i_{M+1}}, 
	\end{split}
\end{align}
and for 
\begin{align*}
	\beta&=\{m_1\dots l_1\}\dots \{m_t \dots l_t, M+1\} \{m_r\dots l_r\},
\\
	\alpha&=\{m_1\dots l_1\}\dots \{m_r\dots l_r\}
\end{align*}
we have 
\begin{align} \label{eq_1.12}
\begin{split}
	\delta(\beta)&=\delta^{i_{M+1}}_{j_{l_t}}\delta(\alpha),\\
	q(\beta)&= q(\alpha) + \sum_{s=t+1}^{r}p_{i_{m_s}j_{l_s}}p_{i_{M+1}j_{M+1}} \\
		&= q(\alpha) +  \sum_{s=t+1}^{r}p_{k_si_{m_s}}p_{i_{M+1}j_{M+1}} +
	 \sum_{s=t+1}^{r}p_{k_sj_{l_s}}p_{i_{M+1}j_{M+1}},\\
	 p_{\overline k}(\beta)&=p_{\overline k}(\alpha) +\sum_{s=t+1}^{r}p_{k_si_{m_s}}p_{k_t j_{M+1}}-\sum_{s=t+1}^r p_{k_s i_{m_s}}p_{k_t j_{l_t}}
	 \\
	 &=p_{\overline k}(\alpha) +\sum_{s=t+1}^{r}p_{k_si_{m_s}}p_{ j_{M+1}j_{l_t}}
	 \end{split}
\end{align}
in view of the definition of $q(\alpha)$, $p_{\overline k}(\beta)$ and $\delta(\beta)$. The presence of  the Kronecker delta 
$\delta_{j_{l_t}}^{i_{M+1}}$ in the coefficient of the corresponding  term  implies that either  this coefficient is zero, or
\[
	q(\beta)+p_{\overline k}(\beta)=p_{\overline k}(\alpha)+q(\alpha)+p^\prime_{\overline k}.
\]
  \end{proof}
  Note that  the image of the restriction of   $\tilde \sigma$ to the subalgebra $U(Q(n))$ of $U(\fgl(n,n))$ lies in $S(Q(n))$ due to  (\ref{eq_1.3}) and the fact that $Q(n)$ is the associative sublagebra of $\fgl(n,n)$. Thus, the map $\tilde \sigma$ also establishes the isomorphism  of   vector spaces $U(Q(n))$ and $S(Q(n))$.
  \section{Duality of Centers}\label{Sec-2}
  Denote by $\fH_k$ the group generated by $a_1,\dots, a_k,\varepsilon$ subject to relations $a^2_1=a^2_2=\dots = a_n^2=\varepsilon$,
  $\varepsilon^2=1$, $a_pa_q=\varepsilon \,  a_qa_p$ for $p,q=1,\dots ,k$, $p\ne q$. Note that $\fH_k$ is a double cover of $\bZ^k_2$. Symmetric group $S_k$ acts on    $\fH_k$  by $\tau a_{p}=a_{\tau(p)}$, $\tau \varepsilon =\varepsilon$.  The {{\it Sergeev group of order $k$}} is the semidirect product of  $\fH_k$ and $S_k$ defined by the above action. 
  Following \cite{Serg-1} we  consider  natural representations of the groups $Se(N)$ and $S_N$ in the tensor space $(\bC^{m|n})^{\otimes N}$. Namely, let $\fA$ be the free associative commutative superalgebra with  homogeneous generators $\{y_i\}$. We introduce  a function $c:\bZ^k_2\times S_k\to  \{\pm 1\}$   by 
  \[
  	 c\, (p(y), \tau) \,y_1\dots y_k= y_{\tau(1)}\dots y_{\tau(k)},
	 \]
where $p(y)=(p(y_1),\dots p(y_n))$ is the vector of parities of $\{y_i\}$. Then for $\tau\in S_N$ we  define its action on $(\bC^{(m|n)})^{\otimes N}$ by 
\[
	 \pi(\tau) (v_1\otimes \dots \otimes v_r)= c\,(\tau^{-1}p(v), \tau) \,v_{\tau^{-1}(1)}\otimes \dots \otimes v_{\tau^{-1}(1)},
	 \]
	 where $v_i\in \bC^{(m|n)}$ are homogeneous elements.
Let 
\begin{align*}
	P=\begin{pmatrix}
	0&1\\
	-1&0
	\end{pmatrix}
	\in \fgl(n,n),
\end{align*}  
and let 
\[
	 \pi(a_k)=1\otimes \dots \otimes P\otimes \dots \otimes 1
	 \]
($P$ is located in the $k$-th place) with  $\pi(\varepsilon)=- \text{Id}$. Then $\pi $ defines  representations of $S_N$   on  $(\bC^{(m|n)})^{\otimes N}$ and $Se(N)$  on  $(\bC^{(n|n)})^{\otimes N}$. Denote  by $\gamma_N$  the vector representations of $\fgl(m,n)$ or $Q(n)$ in the tensor spaces. 

\begin{theorem*}[A.N.\,Sergeev, \cite{Serg-1}]
\begin{align*}
	\pi(S_N)^!&=\gamma_N(U(\fgl(m,n)),\\
	\pi(Se(N))^!&=\gamma_N(U(Q(n))),
\end{align*}
where the symbol $!$ denotes the commutant. 
\end{theorem*}
 This theorem immediately  implies that for any central  element of the  group algebra  of $S_N$ or $Se(N)$ (e.g., for the characteristic function of a conjugacy class) there  exists  a central  element  of the corresponding  universal enveloping algebra  whose action on $(\bC^{m|n})^{\otimes N}$  (or, accordingly, on $(\bC^{n|n})^{\otimes N}$) coincides with the action of the given element. The goal of this section is to find such  elements. 
 
 The following lemma reconstructs the action of an element of $U(\fgl(m,n))$ using its image in differential operators  (cf. Section \ref{Sec-1}). 
  Consider a  scalar product  in the tensor space $(\bC^{(m|n)})^{\otimes N}$  with an orthonormal basis of monomials $\{e_{i_1}\otimes \dots\otimes e_{i_N}\}_{i_1,\dots, i_N\in \{1,\dots, m+n\}}$. 

 \begin{lemma}\label{lem_2.1}
 Let $a\in U(\fgl(m,n))$. Then 
 \begin{align}\label{eq_2.1}
 	(a(e_{j_1}\otimes \dots \otimes e_{j_N}), e_{i_1}\otimes \dots \otimes e_{i_N})
	=(-1)^R\partial(a)(x_{i_1j_1}\dots x_{i_Nj_N})\vert _{(x_{ij})=\text{Id}},
 \end{align}
 where $R$ depends only on $(i_1,\dots, i_N, j_1,\dots, j_N)$.
 \end{lemma}
 \begin{proof}
 It suffices to consider $a= E_{m_1l_1}\dots E_{m_rl_r}$. The result of applying  $E_{m_{r+1-s}l_{r+1-s}}\dots E_{m_rl_r}$
 to a monomial $e_{j_1}\otimes \dots \otimes e_{j_N}$, $s=1,\dots, r-1$,  is a sum
 
  \begin{align}\label{eq_2.2}
 	 \sum\pm e_{k_1}\otimes \dots \otimes e_{k_N}.
   \end{align}
   There is a one-to-one  correspondence between the terms $\pm e_{k_1}\otimes \dots \otimes e_{k_N}$ of the sum (\ref{eq_2.2}) and the ways to obtain   the ordered set $(k_1\dots k_N)$ from the ordered set $(j_1\dots j_N)$ in $s$ steps  replacing $l_{r-t+1}$ with 
   $m_{r-t+1}$ on the $t$-th step, $t=1,\dots, s$. For $s=r$ the value  of $(E_{m_1l_1}\dots E_{m_rl_r}(e_{j_1}\otimes \dots \otimes e_{j_N}), e_{i_1}\otimes \dots \otimes  e_{i_N})$ equals the sum of coefficients  of all monomials of the form $e_{i_1}\otimes \dots \otimes e_{i_N}$
   in  (\ref{eq_2.2}). 
   
Similarly, we  find that 
  \begin{align}\label{eq_2.3}
 	\partial (E_{m_{r+1-s}l_{r+1-s}}\dots E_{m_r l_r})(x_{i_1j_1}\dots x_{i_Nj_N}) =\sum \pm x_{i_1k_1}\dots x_{i_Nk_N},
   \end{align}
  and  as in (\ref{eq_2.2}),  the   quantity $\partial (E_{m_{1}l_{1}}\dots E_{m_{r}l_{r}}) (x_{i_1j_1}\dots x_{i_Nj_N}) |_{(x_{ij})\,=\,\text{Id}}$
  is equal to the sum  of coefficients of the elements $x_{i_1i_1}\dots x_{i_Ni_N}$ in (\ref{eq_2.3}) for $s=r$. By (\ref{eq_1.7}), there is a one-to-one correspondence between  the  summands of (\ref{eq_2.3}) and the same paths of arriving at  $(k_1\dots k_N)$ from $(j_1\dots j_N)$ in $s$ steps. Thus, if $s=r$, for any monomial $e_{i_1}\otimes \dots \otimes e_{i_N}$ in (\ref{eq_2.2}) there is a single element $x_{i_1i_1}\dots x_{i_Ni_N}$ in the right-hand side of  (\ref{eq_2.3}) corresponding to this monomial. 
  
Let us  prove that if $(-1)^p$ is  the coefficient  of  a fixed term $e_{i_1}\otimes \dots \otimes e_{i_N}$  in (\ref{eq_2.2}),  then the coefficient  of the corresponding element $x_{i_1i_1}\dots x_{i_Ni_N}$  equals $(-1)^{p+R}$, where $R$ depends only on $(i_1,\dots, i_N, j_1,\dots, j_N)$. This argument will complete the proof fo  Lemma \ref{lem_2.1} because
 \begin{align*}
 	&\partial (E_{{m_1}l_{1}}\dots E_{m_r l_r})(x_{i_1j_1}\dots x_{i_Nj_N})|_{(x_{ij})=\text{Id}} =\sum (-1)^{p+R}=(-1)^R\sum (-1)^{p}  \\
	&=(-1)^R(E_{{m_1}l_{1}}\dots E_{m_r l_r} (e_{j_1}\otimes \dots e_{j_N}), e_{i_1}\otimes \dots \otimes e_{i_N}).
  \end{align*}
  We fix the  term $e=(-1)^{P^\prime}e_{i_1}\otimes \dots \otimes e_{i_N}$ and the corresponding element $x=(-1)^{P^{\prime\prime}}x_{i_1i_1}
  \dots x_{i_Ni_N}$. Let us construct a sequence of monomials  that traces the appearance of  $e$  in the sum (\ref{eq_2.2}) of the result of the action of  $E_{{m_1}l_{1}}\dots E_{m_r l_r}$ on $e_{j_1}\otimes \dots e_{j_N}$. The sequence consists of $r+1$ monomials:
  \\
  -- the first element is $e_0=e_{j_1}\otimes \dots e_{j_N}$;\\
  -- the last one is $e_r=e$;\\
  -- the $s$-th element is a basis monomial $e_s=(-1)^{a(s)}e_{k_1}\otimes \dots \otimes e_{k_N}$ such that the sum $E_{m_{r+1-s}l_{r+1-s}}e_{s-1}$ contains this element as a  summand.
  
  Note that the monomial $e_s$ differs from the monomial $e_{s+1}$ by a sign and  at most  in one of indices. 
  For $x$ we construct a similar sequence 
  \begin{center}
  $x_0=x_{i_1j_1}\dots x_{i_Nj_N},$\\
  \dots\\
  $x_s=(-1)^{b(s)}x_{i_1k_1}\dots x_{i_Nk_N}$,\\
  \dots\\
  $x_r=x$.
  \end{center}
Let us look at the change  of   coefficients $(-1)^{a(s)}$
 and $(-1)^{b(s)}$  in these sequences. Assume that
 \begin{align*}
 	e_{s-1}&= (-1)^{a(s-1)} e_{k_1}\otimes \dots \otimes e_{k_{t_1}=l_{r-s+1}}\otimes  \dots \otimes e_{k_{t_2}=l_{r-s}}\otimes \dots \otimes e_{k_N},\\
	e_{s}&= (-1)^{a(s)} e_{k_1}\otimes \dots \otimes e_{k_{t_1}=m_{r-s+1}}\otimes \dots \otimes e_{k_{t_2}=l_{r-s}}\otimes \dots \otimes e_{k_N},\\
        e_{s+1}&= (-1)^{a(s+1)} e_{k_1}\otimes \dots \otimes e_{k_{t_1}=m_{r-s+1}}\otimes \dots \otimes e_{k_{t_2}=m_{r-s}}\otimes  \dots \otimes e_{k_N}.
 \end{align*}
 Then 
    \begin{align*}
 	x_{s-1}&= (-1)^{b(s-1)} x_{{i_1}{k_1}}\dots x_{i_{t_1}l_{r-s+1}}\dots x_{i_{t_2}l_{r-s}}\dots x_{i_Nk_N},\\
	x_{s}&= (-1)^{b(s)} x_{{i_1}{k_1}}\dots x_{i_{t_1}m_{r-s+1}}\dots x_{i_{t_2}l_{r-s}}\dots x_{i_Nk_N},\\
	x_{s+1}&= (-1)^{b(s+1)} x_{{i_1}{k_1}}\dots x_{i_{t_1}m_{r-s+1}}\dots x_{i_{t_2}m_{r-s}}\dots x_{i_Nk_N}.\\
		 \end{align*}
At the same time, for 
   \begin{align*}
 	e_{s-1}&= (-1)^{a(s-1)} e_{k_1}\otimes \dots \otimes e_{k_{t_1}=l_{r-s+1}}\otimes \dots \otimes e_{k_{t_2}=l_{r-s}}\otimes  \dots \otimes e_{k_N},\\
	e^\prime _{s}&= (-1)^{a^\prime(s)} e_{k_1}\otimes \dots \otimes e_{k_{t_1}=l_{r-s+1}}\otimes \dots \otimes e_{k_{t_2}=m_{r-s}}\otimes \dots \otimes e_{k_N},\\
        e^\prime_{s+1}&= (-1)^{a^\prime(s+1)} e_{k_1}\otimes \dots \otimes e_{k_{t_1}=m_{r-s+1}}\otimes \dots \otimes e_{k_{t_2}=m_{r-s}}\otimes  \dots \otimes e_{k_N},
 \end{align*}   
 and
 \begin{align*}
 	x_{s-1}&= (-1)^{b(s-1)} x_{{i_1}{k_1}}\dots x_{i_{t_1}l_{r-s+1}}\dots x_{i_{t_2}l_{r-s}}\dots x_{i_Nk_N},\\
	x^\prime_{s}&= (-1)^{b^\prime(s)} x_{{i_1}{k_1}}\dots x_{i_{t_1}l_{r-s+1}}\dots x_{i_{t_2}m_{r-s}}\dots x_{i_Nk_N},\\
	x^\prime_{s+1}&= (-1)^{b^\prime(s+1)} x_{{i_1}{k_1}}\dots x_{i_{t_1}m_{r-s+1}}\dots x_{i_{t_2}m_{r-s}}\dots x_{i_Nk_N}\\
\end{align*}
we obtain
\[
a^\prime(s+1) - a(s+1) =b^\prime(s+1)- b(s+1)=p_{m_{r-s+1}l_{r-s+1}} \, p_{m_{r-s}l_{r-s}}.
\]
This means that the trajectory of replacements   can be reduced to some ``canonical'' form by  interchanging neighboring elements
without   changing the value of  $a(s)-b(s)$. We say that the sequence of monomials $e_0,\dots, e_r$ has a ``canonical'' form if the first $s_1$ elements  of the sequence may have different subscripts  only at the $t_1$-th place, the next $s_2$ elements  may have different subscripts only at  the $t_2$-th place, $\dots$, and the last $s_K$ elements may   have different subscripts only at the $t_K$-th palce, $s_1+s_2+\dots +s_K-1=r$,
$t_1<t_2<\dots<t_K$:
 \begin{align} \label{eq_2.4}
 \begin{split}
 	e_{s_1+\dots +s_d} &=\pm e_{k_1}\otimes \dots \otimes e_{k_ {t_{d}} (1)}\otimes \dots \otimes e_{k_N},\\
 	e_{s_1+\dots+ s_d+1}&=\pm e_{k_1}\otimes \dots \otimes e_{k_{t_{d}} (2)}\otimes \dots \otimes e_{k_N},\\
	\dots 	\\
	 e_{s_1+\dots +s_{d+1}-1}&=\pm e_{k_1}\otimes \dots \otimes e_{k_{t_{d}} (s_{d})}\otimes \dots \otimes e_{k_N}.
	 \end{split}
 \end{align}
 The ``canonical'' form  of a sequence $x_0,\dots, x_{r-1}, x$ is defined similarly. 
 
 Consider a ``canonical''  sequence  and its elements $e_{s_1 +\dots +s_d},\dots , e_{s_1+\dots +s_{d+1}-1 }$.
Note that 
  \begin{align*}
  	 a(s_1+\dots +s_{d+1}-1)= (p_{k_1}+\dots +p_{k_{t_{d+1}-1}})  (p_{l_{r+1-(s_1+\dots +s_d)}} +p_{m_{r+2- (s_1+\dots +s_{d+1}) }}) + a(s_1+\dots +s_d-1),\\
	  b(s_1+\dots +s_{d+1}-1)= (p_{i_1k_1}+\dots +p_{i_{t_{d+1}-1} k_{t_{d+1}-1}})  (p_{l_{r+1-(s_1+\dots +s_d)}} +p_{m_{r+2- (s_1+\dots +s_{d+1}) }}) + b(s_1+\dots +s_d-1).
  \end{align*} 
  Here $p_j=p(e_j)$ is the parity function on $\bC^{m|n}$.
  
  Since $t_1<\dots <t_K$,  for nontrivial contributions we must have that 
    \begin{align*}
   & k_1=i_1,\quad k_2=i_2,\quad \dots ,\quad k_{t_{d+1}-1}= i_{t_{d+1}-1},\\
    & k_{t_{d+1}}(0)= l_{r+1-(s_1+\dots + s_d)}=j_{t_{d+1}},\\
      &  k_{t_{d+1}}(s_{d+1})= m_{r+2-(s_1+\dots + s_{d+1})}=i_{t_{d+1}}.
     \end{align*} 
     Thus, $a(s_1-1)-b(s_1-1)$ depends only on $(i_1,\dots ,i_N, j_1,\dots, j_N)$. Using the induction on $d$ we obtain that $R=a(r)-b(r)$ also depends only on this set. 
 \end{proof}
 
 Let us define an action (not representation!) of $S(\fgl(m,n))$  on $(\bC^{(m|n)})^{\otimes N}$ as follows. For any monomial $E_{m_1l_1} \dots E_{m_rl_r} \in S(\fgl(m,n))$ and a basis  element $e_{j_1}\otimes\dots \otimes e_{j_N}$, we construct a (directed) tree of  sets of  subscripts:
 \\ -- the first element of the tree is $(j_1\dots j_N)$;
 \\ -- on the $s$-th step, $s=1,\dots, r$, the set
 \[
 	(k_1\dots k_{t_1}=l_{r+1-s} \dots k_{t_M}=l_{r+1-s}\dots k_N)
 \]
 generates the sets of subscripts
 \begin{align*}
	 (k_1\dots k_{t_1}&=m_{r+1-s} \dots k_{t_M}=l_{r+1-s}\dots k_N),\\
           &\dots  \\
          (k_1\dots k_{t_1}&=l_{r+1-s} \dots k_{t_M}=m_{r+1-s}\dots k_N),
 \end{align*}
 where the set  $(k_1\dots k_{t_d}=m_{r+1-s} \dots k_N)$
 appears on the $s$-th step of this procedure  if there were no changes on the $t_d$-th place through the steps $1,\dots, s-1$, $1\le t_1<\dots <t_d\le N$.
 
 For every set of subscripts in this tree we attach a coefficient $(\pm 1)$: the sign of  the set  $(k_1\dots k_{t_d}=m_{r+1-s} \dots k_N) $   equals $(-1)^{a+\left( \sum_{q=1}^{t_d-1} p_{k_q}\right)p_{m_{r+1-s }l_{r+1-s}}}$, where   $(-1)^a$  is the sign  of  the preceding set  
$(k_1\dots k_{t_d}=l_{r+1-s} \dots k_N)$ in the tree. 

We define  $(E_{m_1l_1} \dots E_{m_rl_r}(e_{j_1}\otimes \dots \otimes e_{j_N}), e_{i_1}\otimes \dots \otimes e_{i_N})$ to be equal to the sum of  coefficients of the sets of the form $(i_1\dots i_N)$ on the last step in the  tree. 

Note that this action of $S(\fgl(m,n))$  is similar to the action  of $U(\fgl(m,n))$  with the following  difference:  we imposed the condition  that in the  tree every index changes only once, while there can be several  changes at the same place through the action  of $E_{m_1l_1}\dots E_{m_rl_r}\in U(\fgl(m,n))$.
  
  Arguing as in the proof of Lemma \ref{lem_2.1}, we obtain
  \begin{lemma} \label{lem_2.2}
  For every $a\in S(\fgl(m,n))$ we have 
  \begin{align*}
  	(a(e_{j_1}\otimes \dots \otimes e_{j_N}), e_{i_1}\otimes \dots \otimes e_{i_N})=
	(-1)^R\tilde \partial(a) (x_{i_1j_1}\dots x_{i_Nj_N})|_{(x_{ij}=\text{Id})},
  \end{align*}
  where $R$ is as in  Lemma \ref{lem_2.1}.
  \end{lemma}
 \begin{corollary}\label{cor_2.3}
 Any $a\in S(\fgl(m,n))$ and its image $\sigma(a)\in U(\fgl(m, n))$ act on the tensor space identically.
 \end{corollary} 
 \begin{proof}
 The assertion follows from Lemmas \ref{lem_2.1} and \ref{lem_2.2} and the fact that $\tilde \partial =\partial \circ \sigma$.
 \end{proof}
 
 We now introduce the central objects. With a partition $\rho$,  $|\rho|\le N$, we associate an element $a_{\rho, N}\in Z(S_N)$ as follows. Let $\rho_1,\dots, \rho_l$ be the parts of $\rho$, and $(a,\dots ,z)$ stand for  a cyclic permutation  $a\to \dots \to z \to a$. Then, denoting $k=|\rho|=\rho_1+\dots+\rho_l$, we get
  \begin{align*}
 	 a_{\rho, N} =\sum_{\overline i}(i_1,\dots, i_{\rho_1})(i_{\rho_1+1},\dots, i_{\rho_1+\rho_2})\dots(i_{\rho_1+\dots +\rho_{l-1}},\dots, i_{k}),
   \end{align*}
   where the summation is taken over all  sequences of pairwise distinct numbers 
$\overline i = (i_1,\dots, i_k)$, $i_t\in\{1,\dots, N\}, t=1,\dots, k$.
   We also define an element $q_{\rho, N}\in Z(Se(N))$ by
   \[
   q_{\rho, N}=\sum_{\overline i}\sum_\alpha (a_{\alpha_1}\dots a_{\alpha_s})(i_1,\dots, i_{\rho_1})\dots (i_{\rho_1+\dots +\rho_{l-1},} \dots, i_k)
   (a_{\alpha_1} \dots a_{\alpha_s})^{-1},
   \]
   where the sum is taken  over all distinct $\overline i = (i_1,\dots, i_k)$, $i_t\in\{1,\dots, N\}, t=1,\dots, k$,  and all different subsets $\alpha =\{\alpha_1<\alpha_2<\dots <\alpha_s\}$
   of the set $\{i_1,\dots, i_k\}$.
   
   On the other side, let 
   \begin{align*}
	I_k=\sum_{\overline i}  (-1)^{p_{i_2}+\dots +p_{i_k}} E_{i_1 i_2} E_{i_2i_3}\dots E_{i_ki_1} \in S(\fgl(m,n)),
   \end{align*}
   and 
      \begin{align*}
	J_k=\sum_{\overline i}  (-1)^{p_{i_2}+\dots +p_{i_k}} F_{i_1 i_2} F_{i_2i_3}\dots F_{i_ki_1} \in S(Q(n)).
   \end{align*}
   Here $F_{ij}$ is the sum  $E_{ij}+E_{\delta(i)\delta(j)}$ of two  matrix units, where $\delta $  is a parity-swapping involution defined by
   $\delta(i) \equiv  i+n \mod  2n$.
   
   We will prove that $a_{\rho, N}$ and $I_\rho=I_{\rho_1}\dots I_{\rho _l}$ (as well as the image of $I_\rho$ in the universal enveloping algebra 
   under the special symmetrization, cf. Corollary \ref{cor_2.3}) act in the tensor space identically. The same assertion holds for $q_{\rho, N}$ and $J_\rho = J_{\rho _1}\dots J_{\rho_l}$. 
   \begin{theorem}\label{thm_2.4}
   In $(\bC^{m|n})^{\otimes N}$, $a_{\rho, N}\in \bC[S_N]$ and $\sigma (I_\rho)\in U(\fgl(m,n))$ act identically.
   \end {theorem}
   \begin{proof}
   By Corollary \ref{cor_2.3}, for any $a\in S(\fgl(m,n))$ the actions of $a$ and $\sigma(a)\in U(\fgl(m,n))$ in the tensor space coincide. The value of $(I_\rho(e_{j_1}\otimes \dots \otimes e_{j_N}), e_{i_1}\otimes \dots \otimes e_{i_N} )$ equals the sum of coefficients  $(\pm 1)$ of the summands $e_{i_1}\otimes \dots \otimes e_{i_N}$ appearing  after the application of $I_{\rho}$ to $e_{j_1}\otimes \dots \otimes e_{j_N}$. Moreover, the number of these summands equals  the number of  ways 
   to choose    $\rho_1+\dots +\rho_l$ different indices from the sequence $(j_1,\dots, j_N)$, rearrange by  the cyclic shift the first $\rho_1$ chosen indices, do the same for the next $\rho_2$ indices, ..., for the last $\rho_l$ indices, and finally to arrive at $(i_1,\dots, i_N)$.
   
   A  similar conclusion can be made for $(a_{\rho,N}(e_{j_1}\otimes \dots \otimes e_{j_N}),  e_{i_1}\otimes \dots \otimes e_{i_N})$. This means that we have  a  natural one-to-one correspondence between the terms proportional to   $e_{i_1}\otimes \dots \otimes e_{i_N}$  in the sums for 
   $a_{\rho,N}(e_{j_1}\otimes \dots \otimes e_{j_N})$ and    $I_{\rho}(e_{j_1}\otimes \dots \otimes e_{j_N})$.
   
   Let us prove that the corresponding  monomials  have the same coefficients. Suppose $T$ is one of the summands of $I_\rho$:
   \begin{align*}
   	T&= (-1)^{Q(T)} (E_{i_1 i_2} \dots E_{i_{\rho_1} i_1} )\dots  (E_{i_{1+\Sigma_{l-1}} i_{2+\Sigma_{l-1}} } \dots E_{  i_{\Sigma_l}  i_{1+\Sigma_{l-1}  } }),\\
	Q(T)&=p_{i_2}+\dots +p_{i_{\rho_1}}+\dots +p_{i_{2+\Sigma_{l-1}}} +\dots +p_{i_{\Sigma_l}},\\	
   \end{align*}
 where  $ \Sigma_s=\sum_{t=1}^s\rho_t$, and  in particular,  $\Sigma_l=|\rho|=k$.
   Direct computation shows that the coefficient of the monomial
   \[
   	(e_{i_{\rho_1}} \otimes e_{i_{1}} \otimes \dots \otimes e_{i_{\rho_1-1}} )\otimes \dots \otimes 
	(e_{i_{\Sigma_l}} \otimes e_{i_{\Sigma_{l-1}+1}} \otimes \dots \otimes e_{i_{\Sigma_l-1}} )\otimes (e_{i_{N-\Sigma_l}}\otimes \dots  \otimes e_{i_N})
   \]
   in the sum 
      \[
   	T((e_{i_{1}} \otimes \dots \otimes e_{i_{\rho_1}}) \otimes \dots \otimes 
	(e_{i_{1+\Sigma_{l-1}} } \otimes \dots \otimes e_{i_{\Sigma_{l}}} )\otimes \dots \otimes (e_{i_{N-\Sigma_l} }\otimes \dots \otimes e_{i_N}))
	\]
	is equal to the coefficient of the same monomial in the  result of the action of a summand in $a_{\rho, N}$ of the form  $\tau =(1,2,\dots, \rho_1)\dots (1+\Sigma_{l-1},\dots, \Sigma_l) $ on the same initial vector. 
	
	Now we shall prove that  for any basis element $e\in (\bC^{(m|n)})^{\otimes N}$ with the indices  $i_1,\dots, i_{k}$
	on the $\pi(1)$-th, ... $\pi(k)$-th places,  $\pi\in S_N$, the corresponding summands of the results of actions of $I_\rho$ and $a_{\rho, N}$ on this basis element have the same coefficients. 
	
	First, we compare the coefficients  of two monomials with indices   $i_1,\dots, i_{k}$  on the places corresponding to permutations   $\pi$ and $\tilde \pi=(t,t+1)\pi$, $t\in \{1,\dots, N-1\}$. Let us fix a monomial 
	$e_T=\pm e_{i_{(\pi\tau)^{-1}(1)}}\otimes \dots  \otimes e_{i_{ (\pi\tau) ^{-1}(N)} }$ in the sum  $T(e)$, where 
	$e=e_{i_{\pi^{-1}(1)}} \otimes \dots \otimes e_{i_{\pi^{-1}}(N)}$, and  a monomial 
	$\tilde e_T=\pm e_{i_{(\tilde \pi\tau)^{-1}(1)}}\otimes \dots  \otimes e_{i_{ (\tilde\pi\tau) ^{-1}(N)} }$ in $T(\tilde e)$, where 
	$\tilde e=e_{i_{\tilde \pi^{-1}(1)}} \otimes \dots \otimes e_{i_{\tilde \pi^{-1}}(N)}$. The coefficient of  $e_T$
	in $(\pi\tau\pi^{-1})(e)$  is different from the coefficient of  $\tilde e_T$ in $(\tilde \pi\tau\tilde \pi^{-1}) (\tilde e)$ by $(-1)^{\tilde p}$, where 
	\[
		 \tilde p = p_{i_{\pi^{-1}(t)}} p_{i_{\pi^{-1}(t+1)}} +p_{i_{(\pi\tau)^{-1}(t)}} p_{i_{(\pi\tau)^{-1}(t+1)}}.
	\]
To compute the difference of  coefficients of corresponding monomials $e_T$ in $T(e)$ and $\tilde e_T$ in $T(\tilde e)$ we need to consider the following cases:
\begin{enumerate}
\item 
$
T=(-1)^{Q(T)} E_{i_1 i_2}\dots E_{i_{(\pi\tau)^{-1}(t)} i_{\pi^{-1}(t)}} \dots E_{i_{(\pi\tau)^{-1}(t+1)} i_{\pi^{-1}(t+1)}} \dots
 E_{i_{\Sigma_l} i_{1+\Sigma_{l-1}}};
$
\item 
$
T=(-1)^{Q(T)} E_{i_1 i_2}\dots E_{i_{(\pi\tau)^{-1}(t+1)} i_{\pi^{-1}(t+1)}} \dots E_{i_{(\pi\tau)^{-1}(t)} i_{\pi^{-1}(t)}} \dots
 E_{i_{\Sigma_l} i_{1+\Sigma_{l-1}}};
$
\item 
the set $\{i_1,\dots, i_{\Sigma_l}\}$ does not contain $\pi^{-1}(t)$ or $\pi^{-1}(t+1)$.

\end{enumerate}

A straightforward computation shows that in all cases the ratios of the coefficients coincide with $(-1)^{\tilde p}$. Symmetric group $S_N$ is generated  by transpositions  $(t,t+1)$, $t=1,\dots, N-1$, hence any permutation  $\pi$ can be represented as a  product  of transpositions. 
We already proved that the signs  match   when $\pi$ is the unity of $S_N$. This implies that the coefficients of $e_T$ in $T(e)$ and $\pi\tau\pi^{-1}(e)$ coincide  for any $\pi \in S_N$. The proof  is complete. 
   \end{proof}
   \begin{theorem}\label{thm_2.5}
   In $(\bC^{n|n})^{\otimes N}$, $q_{\rho, N}\in \bC[Se(N)]$ and   $\sigma(J_{\rho})\in U(Q(n))$ act identically. 
   \end{theorem}
   \begin{proof}
   We follow the lines of the  proofs of Lemma \ref{lem_2.1} and Theorem \ref{thm_2.4}. Note that both, $(J_\rho(e_{j_1}\otimes \dots \otimes e_{j_N}),e_{i_1}\otimes \dots  \otimes e_{i_N}) $  and $(q_{\rho,N}(e_{j_1}\otimes \dots \otimes e_{j_N}),e_{i_1}\otimes \dots  \otimes e_{i_N}) $, are  the sums  of signs  of all different ways of  reaching    the sequence  $(i_1,\dots, i_N)$ from the sequence $(j_1,\dots, j_N)$  by the following procedure. 
   We choose $k=|\rho|$ different  places  in the sequence  $(i_1,\dots, i_N)$. Then we choose some places among them (called {\it marked places}),  where we change  the parities of the indices by $i\mapsto \delta(i)\equiv i+n\,\text{mod}\,2n$. After that we  apply a permutation of cyclic type $\rho$ to  the indices  on  the chosen $k$ places   and change  the parities  on all  marked places again. Our goal  is to prove that the resulting signs coincide for  $J_\rho$ and $q_{\rho,N}$. 
   
  We shall compare the signs of the monomials   with the ones from the proof of Theorem \ref{thm_2.4},  since  we know that the signs are the same for ``pure'' permutations  without any changes of parities. For $J_\rho$  we have to compare the signs of action of two monomials: the action of the monomial
   \[
   	T_1= E_{l_1l_2}\dots E_{\delta(l_s)\delta(l_{s+1})}\dots E_{\delta(l_t)\delta(l_{t+1})}\dots  E_{l_kl_1}
	   \]
 on the vector 
\[
	e_{j_1}\otimes \dots \otimes e_{j_N}\quad  =\quad e_{j_1}\otimes \dots \otimes e_{\delta(l_{s+1})}\otimes \dots \otimes e_{\delta(l_{t+1})}\otimes  \dots \otimes e_{j_N},
\]	   
where  the indices $\delta(l_{t+1})$ and $\delta(l_{s+1})$ are on marked places,  with the resulting vector 
\[
	e_{i_1}\otimes \dots \otimes e_{i_N} \quad =\quad e_{i_1}\otimes \dots \otimes e_{\delta(l_{s})}\otimes \dots \otimes e_{\delta(l_{t})}\otimes  \dots \otimes e_{i_N},
\]	
and the action of the monomial 
  \[
   	T_2= E_{l_1l_2}\dots E_{l_s l_{s+1}}\dots E_{l_t l_{t+1}}\dots  E_{l_kl_1}
	   \]
on the vector 
  \[
	e_{j_1}\otimes \dots\otimes e_{l_{s+1}} \otimes \dots \otimes e_{l_{t+1}}\otimes \dots \otimes e_{j_N} 
	\]
with the resulting vector
   \[
	e_{i_1}\otimes \dots\otimes e_{l_{s}} \otimes \dots \otimes e_{l_{t}}\otimes \dots \otimes e_{i_N}.
	\]	
	The difference in signs in these two cases  accumulates  only on the marked places where all the vectors  changed their parities. Therefore, the difference in the signs of the actions by $T_1$  and $T_2$  is the sign that we obtain by brining  the matrix  units through the marked places. This difference equals $s_1=(-1)^{ \sum p_{i_k} +p_{j_k}}$, where the sum is taken over all (unmarked) $k$-th places  that  have   an odd number of the marked  places to the left from them. 

We now turn to the element $q_{\rho, N}$.  Let us compute the difference between the sign of the action of the element $(a_{\alpha_1}\dots a_{\alpha_l})\pi(a_{\alpha_1}\dots a_{\alpha_l})^{-1}$ and the sign  of the action  of $\pi \in S_N$. The marked places are exactly $\alpha_1,\alpha_2, \dots, \alpha_l$. We can assume that $\alpha_1>\alpha_2>\dots>\alpha_l$. When $a_{\alpha_s}$ changes  $e_{j}$
to $e_{i=\delta(j)}$ on the $\alpha_s$-th place, we obtain the sign  $(-1)^{ p_j=p_i+1}$ from the definition of the operator $P$. Note that 
\[
(a_{\alpha_1}\dots a_{\alpha_l})\pi(a_{\alpha_1}\dots a_{\alpha_l})^{-1}
=\varepsilon^l\, (a_{\alpha_1}\dots a_{\alpha_l})\pi(a_{\alpha_l}\dots a_{\alpha_1}).
\]
Thus, the action of the two products of $a_\alpha$'s on the marked places  at the beginning and at the end  gives
$s_2=(-1)^{\sum(p_{i_{\alpha_k}}+p_{j_{\alpha_k}})+l}$.  Applying the action of the product  $(a_{\alpha_1}\dots a_{\alpha_l})$ by ``carrying  the operator $P=\begin{pmatrix}
0&1\\
-1&0
\end{pmatrix}$
through the odd vectors''  we obtain $s_3=(-1)^{\sum p_{j_{k}}} $, where the sum is taken over the $k$-th places such that  to the right from them there is an odd number of the marked places. Similarly,  by the action of the product 
$(a_{\alpha_1}\dots a_{\alpha_l})$ in the end we obtain $s_4=(-1)^{\sum p_{i_{k}} }$, where  the sum is taken over  the same $k$'s as in the previous case. It remains to note that $s_0s_1s_2s_3s_4=(-1)^{\sum (p_{i_{k}} +p_{j_k}) }=1$, where the sum is taken over all $k$ and
$s_0=(-1)^l$ is the sign  of $\varepsilon ^l$. The proof is complete. 
   \end{proof}


\end{document}